\documentclass[12pt]{article}

\usepackage{amsmath}
\usepackage{amsfonts}
\usepackage{amssymb}
\usepackage{latexsym}
\usepackage{geometry}

\newtheorem{theorem}{Theorem}[section]
\newtheorem{lemma}{Lemma}[section]
\newenvironment{proof}{\noindent{\bf Proof:}}{\hfill\fbox{}\vspace*{1mm}}
\newtheorem{prop}[theorem]{\bf Proposition}
\newtheorem{coro}[theorem]{\bf Corollary}

\newtheorem{remark}[theorem]{\bf Remark}

\newtheorem{defn}[theorem]{\bf Definition}

\newcommand{\ulxi} {{\underline \xi}}
\newcommand{\ulx} {{\underline x}}
\newcommand{\uly} {{\underline y}}

\newcommand{\bfe} {{\mathbf e}}

\AtEndDocument{\bigskip{\footnotesize
\textsc{Faculty of Information Technology, Macau University of Science and Technology, Macao, China}\par
\textit{E-mail address: }\texttt{pdang@must.edu.mo}\par
\medskip
\textsc{School of Mathematics (Zhuhai), Sun Yat-Sen University (Zhuhai), China}\par
\textit{E-mail address: }\texttt{maiweixiong@gmail.com}\par
\medskip
\textsc{Macau Institute of Systems Engineering, Macau University of Science and Technology, Macao, China}\par
\textit{E-mail address: }\texttt{tqian@must.edu.mo}
}}
\usepackage{lipsum}
\begin{document}
\title{Fourier Spectrum of Clifford $H^{p}$ Spaces on $\mathbf{R}^{n+1}_+$ for $1\leq p \leq \infty$\let\thefootnote\relax\footnotetext{This work was supported by the Science and Technology Development Fund, Macao SAR: 0006/2019/A1, 154/2017/A3; NSFC Grant No. 11701597; NSFC Grant No. 11901594; The Science and Technology Development Fund, Macao SAR: 0123/2018/A3.}}
\author{Pei Dang, Weixiong Mai\thanks{Corresponding author}, Tao Qian}

\date{}
\maketitle
\begin{center}
\begin{minipage}{120mm}
\begin{center}{\bf Abstract}\end{center}
{This article studies the Fourier spectrum characterization of functions in the Clifford algebra-valued Hardy spaces $H^p(\mathbf R^{n+1}_+), 1\leq p\leq \infty.$ Namely, for $f\in L^p(\mathbf R^n)$, Clifford algebra-valued, $f$ is further the non-tangential boundary limit of some function in $H^p(\mathbf R^{n+1}_+),$ $1\leq p\leq \infty,$ if and only if $\hat{f}=\chi_+\hat{f},$ where $\chi_+(\underline{\xi})=\frac{1}{2}(1+i\frac{\ulxi}{|\ulxi|}),$ the Fourier transformation and the above relation are suitably interpreted (for some cases in the distribution sense). These results further develop the relevant context of Alan McIntosh. As a particular case of our results, the vector-valued Clifford Hardy space functions are identical with the conjugate harmonic systems in the work of Stein and Weiss. The latter proved the corresponding singular integral version of the vector-valued cases for $1\leq p<\infty.$ We also obtain the generalized conjugate harmonic systems for the whole Clifford algebra-valued Hardy spaces rather than only the vector-valued cases in the Stein-Weiss setting.

}
\end{minipage}
\end{center}
\begin{center}
{\large \it In memory of Alan McIntosh}
\end{center}

\section{Introduction}
The classical Paley-Wiener Theorem
  asserts that for a $L^2(\mathbf R)$-function $f$, scalar-valued, it is further the non-tangential boundary limit (NTBL) of a function in the Hardy $H^2$ space in the upper half plane if and only if the Fourier transform of $f$, denoted by $\hat{f}$, satisfies the relation $\hat{f}=\chi_+\hat{f},$  where $ \chi_+$ is the characteristic (indicator) function of the set $(0,\infty),$ that takes the value $1$ when its argument is in $(0,\infty)$ and otherwise zero. This amounts to saying that a characteristic property of boundary limit functions of the Hardy space functions is that their Fourier transforms vanish at the negative spectra. This classical Fourier spectrum characterization of the Hardy $H^2$ space functions has been studied and generalized by different authors (see, for instance, \cite{Duren,Garnett}). Among recent studies \cite{Q1} and \cite{QXYYY} give a throughout treatment to the analogous results for $L^p(\mathbf R)$ for $p\in [1,\infty]$. Those papers prove that if $f$ is a function in $L^p(\mathbf R), 1\leq p\leq \infty,$ then $f$ is the NTBL of some function in the Hardy $H^p$ space in the upper half plane if and only if the Fourier transform $\hat{f}$ satisfies the relation $(\hat{f}, \psi)=0$ for $\psi$ being any function in the Schwartz class whose support lays in the closure of left half of the real line. We note that for $2<p\leq \infty$ this last statement is interpreted as that $\hat{f}$, as a distribution, is supported in $[0,\infty)$. The Fourier spectrum characterization results have implications to the Hilbert transform characterizations of the Hardy space functions, as well as to Hardy space decompositions of $L^p$ functions, the latter being through the Fourier spectrum decomposition. We note that the Hardy spaces decomposition can also be extended to the $L^p$ spaces for $0<p<1$ (\cite{Deng-Qian}) although there do not exist spectrum decomposition results.

In higher dimensional Euclidean spaces there exist analogous results. The above mentioned Fourier spectrum, Hilbert transformations, and the Hardy space decompositions are all based on the Cauchy type complex structure associated with the underlying domain on which the Hardy space functions are defined. In $\mathbf R^n$, $n>2$, there are two distinguished complex structures, of which one is several complex variables and the other is Clifford algebra. Both those complex structures in relation to their respective Hardy spaces are treated in \cite{SW} and \cite{Stein2}. The several complex variables setting corresponds to the Hardy spaces on tubes. The Clifford algebra setting corresponds to the conjugate harmonic systems.

Fourier spectrum properties of Hardy spaces on tubes were first studied in \cite{SW} and \cite{Stein2} with the restriction on $p=2.$ Certain one-way results for $H^p(T_\Omega)$ and a partial range of the space index $p$, where $\Omega$ is an irreducible symmetric cone and $T_{\Omega}=\{x+iy\in \mathbf C^n; x\in \mathbf R^n,y\in\Omega\subset \mathbf R^n\}$, were obtained in \cite{Garrigos}.
In \cite{Hor} H\"ormander proved some results corresponding to the type of Paley-Wiener Theorem for bandlimited functions involving entire functions in several complex variables. Fourier spectrum characterizations of the Hardy spaces on tubes for all cases $1\leq p\leq \infty$ are thoroughly studied in \cite{Li-Deng-Qian}.

The present paper gives Fourier spectrum characterizations for functions in the Clifford algebra-valued Hardy spaces for the whole range $p\in [1,\infty]$. As a particular case, the vector-valued case corresponding to the conjugate harmonic systems was previously and fundamentally studied in \cite{Stein2,SW}, and further in \cite{Gilbert-Murray,Mitrea,Kou-Qian}. The previous studies, besides the restriction to vector-values, were also restricted to the singular integral version, and the index range is restricted to $1\leq p<\infty$. The main results of this study imply the Hilbert transformation eigenvalue characterizations of the Hardy spaces and the Hardy spaces decompositions of the $L^p$ functions.

The crucial notion with the Clifford algebra setting of the Euclidean spaces is the projection functions $\chi_\pm$ defined by

\[ \chi_\pm (\ulxi)=\frac{1}{2}\left(1\pm i\frac{\ulxi}{|\ulxi|}\right)\]

\noindent and the associated generalizations of the trigonometrical exponential function

\[ e^\pm (x, \underline{\xi})=e^{2\pi i \langle \underline{x}, \underline{\xi}\rangle}e^{\mp 2\pi x_0|\underline{\xi}|}\chi_\pm (\ulxi),
\]
where $x=x_0+\ulx.$

The purpose of this paper is to declare the Fourier multiplier form of the Clifford Cauchy integral representation formula for the Clifford algebra-valued Hardy $H^p$ functions $F$ in the upper-half space $\mathbf R^{n+1}_+$ for all $p\in [1,\infty]$:

\[ F(x_0+\underline{x})=\int_{\mathbf R^n}e^{2\pi i \langle\underline{x},\underline{\xi}\rangle}e^{-2\pi x_0|\underline{\xi}|}\chi_+(\underline{\xi})\hat{F}(\underline{\xi})d\underline{\xi}.
\]
The above formula is also the Laplace transform of functions on $\mathbf R^n$ in the Clifford algebra setting provided that the formula makes sense.
We note that when $1\leq p\leq 2,$ the above relation is valid in the Lebesgue integration sense, while when $p>2$ it is valid in the distribution sense. We will prove that, for all $p\in [1,\infty],$ a Clifford algebra-valued $L^p(\mathbf R^n)$-function $F$ satisfies $F\in H^p(\mathbf R^{n+1}_+)$ (the NTBL of some Clifford algebra-valued Hardy space function) if and only if $\hat F=\chi_+\hat F,$ where the multiplication between $\chi_+$ and the distribution $\hat F$ will be precisely defined in the rest part of the paper; and $F\in H^p(\mathbf R^{n+1}_+)$ if and only if $HF=F,$ where $H=-\sum_{k=1}^n \bfe_k R_k$ is the Hilbert transformation, and $R_k,k=1,...,n$ are the Riesz transformations. For the Clifford algebra-valued Hardy spaces in the lower-half of the space $\mathbf R^{n+1}$ we have the counterpart results $\hat F=\chi_-\hat F;$ and $HF=-F.$

 In \cite{Stein2} (see also Propositions \ref{stein p65}-\ref{stein p221 th6}) the characterization by the Riesz transformations of the NTBLs of the vector-valued Hardy spaces functions, or alternatively the conjugate harmonic systems, for $1\leq p<\infty,$ is proved.
 When $p=2$, the corresponding Fourier spectrum characterization of the vector-valued functions in $H^2(\mathbf R^{n+1}_+)$ can be directly derived from Theorem 3.1 in \cite[page 65]{Stein2} (see also Proposition \ref{stein p65}) through the relation $\hat F=(\hat f_0, (R_1(f_0))^\wedge,...,(R_n(f_0))^\wedge)=(\hat f_0, -i\frac{\xi_1}{|\ulxi|}\hat f_0,...,-i\frac{\xi_n}{|\ulxi|}\hat f_0).$ Our study systematically treats the Fourier multiplier aspect for the Clifford algebra-valued functions case, and for the whole index range $1\leq p\leq \infty,$ involving distributional and $BMO$ functional analysis.

The paper is organized as follows. In \S 2 some notations and terminologies are given. In \S 3 we prove the main results. In \S 4, as an application of the Fourier spectrum of $H^p(\mathbf R^{n+1}_+)$, we prove the analogous result in the Clifford algebra-valued Bergman spaces on $\mathbf R^{n+1}_+$.

\section{Preliminaries}
 \subsection { $\mathcal H^{p}$ space in terms of conjugate harmonic systems}
Let $\mathbf R_+^{n+1}=\{x=(\ulx,x_0)\in \mathbf R^{n+1};x_0>0,\ulx\in \mathbf R^n\},$ and $u_j,0\leq j\leq n,$ be functions defined on $\mathbf R^{n+1}_+.$ Suppose $F=(u_0,u_1,\ldots,u_n)$ satisfies the generalized Cauchy-Riemann systems in $\mathbf{R}_+^{n+1}$, i.e.,
\begin{align}\label{GCR}
\begin{split}
 \sum_{j=0}^n\frac{\partial u_j}{\partial x_j}=0, &\\
\frac{\partial u_j}{\partial x_k}=\frac{\partial u_k}{\partial x_j}, & \quad j\neq k, 0\leq j,k\leq n.
\end{split}
\end{align}
Such $(n+1)$-tuple $F=(u_0,u_1,\ldots,u_n)$ is called a conjugate harmonic system (\cite{Stein2,SW}).
 For $0<p<\infty,$ we say that $F\in \mathcal H^p(\mathbf R^{n+1}_+),$ the vector-valued monogenic $p$-Hardy space, if $F$ satisfies (\ref{GCR}), and moreover,
\begin{eqnarray}\label{hp bounded condition}
||F||_p=\sup\limits_{x_0>0}\left[\int_{\mathbf{R}^n}|F(\ulx,x_0)|^pd\ulx\right]^{\frac{1}{p}}<\infty.
\end{eqnarray}
For $p=\infty,$ we say that $F\in \mathcal H^\infty(\mathbf R^{n+1}_+)$ if $F$ satisfies (\ref{GCR}), and
\begin{eqnarray}\label{hp bounded condition}
||F||_\infty=\sup\{|F(\ulx,x_0)| \ : \ (\ulx,x_0)\in\mathbf R^{n+1}_+ \}<\infty.
\end{eqnarray}
$||F||_p$ is a norm when $1\leq p\leq\infty,$ and a Hilbert space norm when $p=2$ (see c.f. \cite[page 220]{Stein2}).  For $\mathcal H^p(\mathbf R^{n+1}_+),$ the following results for harmonic functions are well-known.
 \begin{prop}[{\cite[ page 78]{Stein2}}]\label{stein p78}
Let $u(\ulx,x_0)$ be harmonic in $\mathbf{R}_+^{n+1}.$\\
\noindent One has:\\
{\rm (a)} If $1<p\leq\infty, $ $u(\ulx,x_0)$ is the Poisson integral of an $L^p(\mathbf{R}^n)$ function if and only if $\sup\limits_{x_0>0}\|u(\ulx,x_0)\|_p<\infty.$

\noindent{\rm (b)} $u(\ulx,x_0)$ is the Poisson integral of a Borel measure if and only if $\sup\limits_{x_0>0}\|u(\ulx,x_0)\|_1<\infty.$
\end{prop}
\begin{prop}[{\cite[Theorem 1.3, page 62]{Deng-Han}}]
For $f\in L^p(\mathbf R^n),1\leq p\leq \infty,$ if $u(\ulx,x_0)$ is the Poisson integral of $f$, then\\
{\rm (a)} $u^*(\ulx)=\sup_{|\uly-\ulx|<\alpha x_0}|u(\uly,x_0)|\leq A\mathcal M(f)(\ulx),$ where $\alpha>0,$ $\mathcal M(f)$ is the Hardy-Littlewood maximal function of $f.$\\
{\rm (b)} For almost all $\ulx\in \mathbf R^n,$ one has
\begin{align*}
\lim_{(\uly,x_0)\to (\ulx,0),|\uly-\ulx|<\alpha x_0}u(\uly,x_0)=f(\ulx).
\end{align*}
\end{prop}

 \begin{prop}[{\cite[page 65]{Stein2}}]\label{stein p65}
Let $f_0$ and $f_1,\ldots,f_n$ belong to $L^2(\mathbf{R}^n),$ and let their respective Poisson integrals be $u_0(\ulx,x_0)=P_{x_0}*f_0, u_1(\ulx,x_0)=P_{x_0}*f_1,\ldots,u_n(\ulx,x_0)=P_{x_0}*f_n.$ Then a necessary and sufficient condition that \[f_j=R_j(f_0),\ j=1,\ldots,n,\]
is that $(u_0,...,u_n)$ satisfies the generalized Cauchy-Riemann equations (\ref{GCR}), where $R_j$ is the $j$-th Riesz transformation, that is,
\begin{align}\label{Riesz-transform}
f_j(\ulx)=R_j(f_0)(\ulx)=\lim_{\epsilon\to 0}\frac{1}{\sigma_n}\int_{|\ulx-\uly|>\epsilon}\frac{x_j-y_j}{|\ulx-\uly|^{n+1}}f_0(\uly)d\uly
\end{align}
with $\sigma_n=\frac{\pi^\frac{n+1}{2}}{\Gamma(\frac{n+1}{2})}.$
\end{prop}

\begin{prop}[{\cite[page 220]{Stein2}}]\label{stein p220}
Suppose that $F\in \mathcal H^p(\mathbf R^{n+1}_+),1<p<\infty.$ Then there exist $f_0,f_1,f_2,\ldots, f_n,$ each in $L^p(\mathbf{R}^n),$ so that $u_j(x,x_0)$ is the Poisson integral of $f_j,j=0,\ldots,n.$ Also $f_j=R_j(f_0),$ $R_1,R_2, \ldots,R_n$ are the Riesz transformations. Conversely, suppose $f_0\in L^p(\mathbf{R}^n),$ and let $f_j=R_j(f_0),$ and $u_j(x,x_0)$ be the Poisson integrals of $f_j,j=0,\ldots,n.$ Then $F=(u_0,u_1,\ldots,u_n)\in \mathcal H^p(\mathbf R^{n+1}_+);$ moreover $\|f_0\|_{L^p}\leq\|F\|_p\leq A_p\|f_0\|_{L^p}.$
\end{prop}

 \begin{prop}[{\cite[page 221]{Stein2}}]\label{stein p221 th6}
Let $F\in \mathcal H^1(\mathbf R^{n+1}_+).$ Then $\lim\limits_{x_0\to 0}F(\ulx,x_0)=F(\ulx)$ exists almost everywhere, as well as in the $L^1(\mathbf{R}^n)$ norm. Also
\[\int_{\mathbf{R}^n}\sup\limits_{x_0>0}|F(\ulx,x_0)|d\ulx\leq A\sup\limits_{x_0>0}\int_{\mathbf{R}^n}|F(\ulx,x_0)|d\ulx=A\|F\|_1. \]
Moreover, the space $\mathcal H^1(\mathbf R^{n+1}_+)$ is naturally isomorphic with the space of $L^1(\mathbf{R}^n)$ functions $f_0$ which have the property that $R_j(f_0)\in L^1(\mathbf{R}^n),j=1,\ldots,n.$ The $\mathcal H^1$ norm is then equivalent with $\|f_0\|_1+\sum\limits_{j=1}^n\|R_j(f_0)\|_1.$
Conversely, suppose $f_0\in L^1(\mathbf{R}^n)$ and $R_j(f_0)\in L^1(\mathbf{R}^n), R_j(f_0)=f_j,j=1,\ldots,n.$ Then
\[\sum\limits_{j=0}^n\int_{\mathbf{R}^n}\sup\limits_{x_0>0}|u_j(\ulx,x_0)|d\ulx\leq A\sum\limits_{j=0}^n\|f_j\|_{L^1},\]
where $u_j=P_{x_0}*f_j,j=0,...,n.$
\end{prop}
It is noted that the Riesz transformations $R_j,1\leq j\leq n,$ defined as (\ref{Riesz-transform}) are bounded operator from $L^p(\mathbf R^n)$ to $L^p(\mathbf R^n),1<p<\infty,$ and weakly bounded from $L^1(\mathbf R^n)$ to $L^1(\mathbf R^n).$ When $p=\infty,$ we need to revise the Riesz transforms (see e.g. \cite{Uchiyama}) $R_j,1\leq j\leq n,$ as
\begin{align*}
R_j(f_0)(\ulx)=\lim_{\epsilon\to 0}\frac{1}{\sigma_n}\int_{|\ulx-\uly|>\epsilon}\left(\frac{x_j-y_j}{|\ulx-\uly|^{n+1}}-\chi_{\{|\uly|>1\}}(\uly)\frac{-y_j}{|-\uly|^{n+1}}\right)f_0(\uly)d\uly,
\end{align*}
which are well-defined for $f_0\in L^\infty(\mathbf R^n)$ up to a constant, and $\chi_{\{|\uly|>1\}}$ is the characteristic function for $\{\uly\in \mathbf R^n; |\uly|>1\}$. The revised Riesz transformations are bounded from $L^\infty(\mathbf R^n)$ to $BMO(\mathbf R^n)$ (Bounded Mean Oscillation space).

\def\bRn{\mathbf R^n}

\subsection{Clifford algebra-valued Hardy $H^p$ spaces}
Let $\bfe_1,...,\bfe_n$ be basic elements satisfying
$$
\bfe_j \bfe_k+ \bfe_k \bfe_j =-2\delta_{jk}, \quad j,k=1,...,n,
$$
where $\delta_{jk}$ is the Kronecker delta function.
Let $\mathbf R^n=\{\ulx =x_1\bfe_1+\cdots +x_n\bfe_n;x_j\in \mathbf R, 1\leq j\leq n\}$ be identical with the usual Euclidean space $\mathbf R^n,$ and $\mathbf R^{n+1}_+=\{x_0+\ulx; x_0>0,\ulx\in \mathbf R^n\}.$

The real (complex) Clifford algebra $\mathbf R^{(n)}$ $(\mathbf C^{(n)})$ generated by $\bfe_1,...,\bfe_n,$ is the associative algebra generated by $\bfe_1,...,\bfe_n$ over the real (complex) field $\mathbf R$ $(\mathbf C)$.
The elements of $\mathbf R^{(n)}$ $(\mathbf C^{(n)})$ are of the form $x=\sum_{T}x_T \bfe_T,$ where $T=\{1\leq j_1<j_2<\cdots<j_l\leq n\}$ runs over all ordered subsets of $\{1,...,n\}$, $x_T\in \mathbf R$ $(\mathbf C)$ with $x_{\emptyset}=x_0,$ and $\bfe_T=\bfe_{j_1}\bfe_{j_2}\cdots \bfe_{j_l}$ with the identity element $\bfe_{\emptyset}=\bfe_0=1.$ Sc $x:=x_0$ and NSc $x:=x-\text{Sc }x$ are respectively called the scalar part and the non-scalar part of $x.$
In this paper, we denote the conjugate of $x\in\mathbf C^{(n)}$ by $\overline x=\sum_{T}\overline x_T \overline \bfe_T$, where $\overline \bfe_T=\overline \bfe_{j_l}\cdots \overline \bfe_{j_2}\overline \bfe_{j_1}$ with $\overline \bfe_0=\bfe_0$ and $\overline \bfe_j=-\bfe_j$ for $j \neq 0.$ The norm of $x\in \mathbf C^{(n)}$ is defined as $|x|:=(\text{Sc } \overline x x)^\frac{1}{2}=(\sum_{T}|x_T|^2)^{\frac{1}{2}}.$ Generally, a Clifford algebra is not a division algebra, unless $n=2$, that corresponds to the algebra of quaternions. $x=x_0+\ulx\in \mathbf R^{n+1}$ is called a vector or a para-vector, and the conjugate of a vector $x$ is $\overline x=x_0-\ulx.$ If $x$ is a vector then $x^{-1}=\frac{\overline x}{|x|^2}.$ For more information about Clifford algebra, we refer to \cite{Brackx-Delanghe-Sommen}.

A Clifford algebra-valued function $F$ is left-monogenic (resp. right-monogenic) if
\begin{align*}
DF=(\sum_{k=0}^n \partial_k\bfe_k)F=0\ \left(\text{resp. } FD=F(\sum_{k=0}^n \partial_k\bfe_k)=0\right),
\end{align*}
where $\partial_k=\frac{\partial}{\partial x_k},0\leq k\leq n,$ and $D$ is the Dirac operator. Note that $\overline D(DF)=\Delta F=0$ if $F$ is left-monogenic, which means that each component of $F$ is harmonic.
A function that is both left- and right-monogenic is called a monogenic function. Vector-valued left-monogenic functions are simultaneously right-monogenic functions, and vice-versa, and thus they are monogenic.

The Fourier transform of a function in $L^1(\mathbf R^n)$ is defined as
$$
\hat f(\ulxi)=\mathcal F(f)(\ulxi)=\int_{\mathbf R^n}e^{-2\pi i\langle \ulx,\ulxi\rangle}f(\ulx)d\ulx,
$$
where $\ulxi=\xi_1 \bfe_1+\cdots +\xi_n \bfe_n\in\mathbf R^n,$ and the inverse Fourier transform is formally defined as
$$
 g^\vee(\ulx)=\mathcal F^{-1}(g)(\underline x)=\int_{\mathbf R^n}e^{2\pi i\langle \underline x,\underline \xi \rangle}g(\underline \xi)d\underline \xi.
$$

  We note that the Fourier transformation is linear and thus it, together with some of its properties, can be extended to Clifford algebra-valued functions. In particular, the Plancherel Theorem holds for Clifford algebra-valued functions: For Clifford algebra-valued functions $f, g\in L^2(\bRn)$ there holds
  \[ \int_{\bRn} f(\ulx)\overline{g}(\ulx)d\ulx = \int_{\mathbf R^n} \hat{f}(\ulxi)\overline{\hat{g}}(\ulxi)d\ulxi.\]
  An alternative form of the Plancherel Theorem is
  \[ \int_{\mathbf R^n} {f}(\ulx){g}(\ulx)d\ulx = \int_{\mathbf R^n} \hat{f}(\ulxi){g}^\vee(\ulxi)d\ulxi.\]

\noindent Define, for $x=x_0+\ulx,$
$$
e(x,\underline \xi)=e^+(x,\ulxi)+e^-(x,\ulxi)
$$
with
$$
e^{\pm}(x,\ulxi)=e^{2\pi i\langle \ulx,\ulxi\rangle}e^{\mp 2\pi x_0|\ulxi|}\chi_{\pm}(\ulxi)
$$
(see e.g. \cite{Li-Mc-Qian}),
where $\chi_\pm(\ulxi)=\frac{1}{2}(1\pm i\frac{\ulxi}{|\ulxi|}).$ $\chi_\pm$ enjoy the following projection-like properties:
\begin{align}\label{important}
\chi_-\chi_+=\chi_+\chi_-=0,\quad \chi^2_\pm=\chi_\pm,\quad\chi_++\chi_-=1.
\end{align}

\begin{defn}\label{Hp definition}
Let $F(x)=\sum_{T}f_T(x)\bfe_T,$ where $x=x_0+\underline{x}\in \mathbf{R}^{n+1}_{+}.$ If $F$ is left-monogenic on $\mathbf{R}^{n+1}_{+}$ and satisfies
\begin{eqnarray}\label{hp bounded condition clifford}||F||_{H^p}^p=\sup\limits_{x_0>0}\int_{\mathbf{R}^n} |F(x_0+\underline{x})|^pd\underline{x}< \infty,\ 1\leq p<\infty,\end{eqnarray}
then we say $F(x)$ belongs to the Hardy space $H^p(\mathbf{R}^{n+1}_{+}). $  If $F$ is left-monogenic on $\mathbf{R}^{n+1}_{+}$ and satisfies
\[||F||_{H^\infty}=\sup\limits_{x\in \mathbf{R}^{n+1}_{+}}|F(x)|< \infty,\] then we say $F(x)\in H^{\infty}(\mathbf{R}^{n+1}_{+}).$
\end{defn}
When $p=2,$ the inner product of $H^2(\mathbf R^{n+1}_+)$ is defined as
\begin{align*}
\langle F, G\rangle=\int_{\mathbf R^n} \overline G(\ulx) F(\ulx) d\ulx,
\end{align*}
where $F(\ulx)$ and $G(\ulx)$ are respectively the NTBL functions of $F$ and $G.$ The norm of $H^2(\mathbf R^{n+1}_+)$ is, in fact, equal to
\begin{align*}
||F||_{H^2}^2=\text{Sc }  \langle F, F\rangle.
\end{align*}

In particular, if $F\in H^p(\mathbf R^{n+1}_+),1\leq p\leq\infty,$ is a vector-valued function, i.e., $F(x)=f_0(x)+\sum_{j=1}^n f_j(x)\bfe_j,$ then $F$ corresponds to an element in $\mathcal H^p(\mathbf R^{n+1}_+)$ (\cite{Kou-Qian}).
 To see this, we let $U(x_0+\underline{x})=u_0(x_0+\underline{x})
 -u_1(x_0+\underline{x})\mathbf{e}_1-\cdots-u_n(x_0+\underline{x})\mathbf{e}_n.$ With $\mathbf{e}_0=1,$ we have
\begin{eqnarray}\label{monogenic condition}
DU &=&(\sum_{k=0}^n \partial_k\mathbf{e}_k)[u_0(x_0+\underline{x})\mathbf{e}_0
-u_1(x_0+\underline{x})\mathbf{e}_1-\cdots-u_n(x_0+\underline{x})\mathbf{e}_n]\nonumber\\
&=&\frac{\partial u_0}{\partial x_0}+\frac{\partial u_1}{\partial x_1}+\cdots+\frac{\partial u_n}{\partial x_n}+
\sum_{0\leq j< k\leq n}(\frac{\partial u_j}{\partial x_k}-\frac{\partial u_k}{\partial x_j})\mathbf{e}_j\mathbf{e}_k,
\end{eqnarray}
which means that $DU=0$ if and only if $U$ satisfies (\ref{GCR}).
Thus $\mathcal H^p(\mathbf R^{n+1}_+)$ can be regarded as a proper subspace of $H^p(\mathbf R^{n+1}_+).$

By Propositions \ref{stein p65}$-$\ref{stein p221 th6}, we have, for $1\leq p<\infty,$
\begin{align*}
U(\ulx)=u_0(\ulx)-\sum_{j=1}^n R_j(u_0)(\ulx)\bfe_j.
\end{align*}
For general vector-valued $F\in H^p(\mathbf R^{n+1}_+),1\leq p<\infty,$ there holds
\begin{align*}
F(\ulx)=f_0(\ulx)-\sum_{j=1}^n R_j(f_0)(\ulx)\bfe_j=(I+H)f_0(\ulx),
\end{align*}
where the operator $H,$ the Hilbert transformation, is defined by
\begin{align*}
H=-\sum_{j=1}^n R_j\bfe_j.
\end{align*}

\noindent Note that the Fourier multiplier of $R_j$ is $-i\frac{\xi_j}{|\ulxi|}$ (see e.g. \cite{Grafakos}), i.e.,
\begin{align*}
R_j(f_0)(\ulx)=\left(-i\frac{\xi_j}{|\ulxi|}\hat f_0(\ulxi)\right)^\vee (\ulx).
\end{align*}
This implies that the Fourier multipliers of $\frac{1}{2}(I\pm H)$ are, respectively, $\chi_{\pm}(\ulxi).$

The Cauchy integral formula in Clifford analysis may be regarded as the reproducing formula. In particular, for $F\in H^p(\mathbf R^{n+1}_+),$ the following reproducing formula plays an important role.
\begin{prop}[see e.g. \cite{Gilbert-Murray,Mitrea}]\label{Cauchy-formula}
For $F\in H^p(\mathbf R^{n+1}_+),1\leq p<\infty,$ we have
\begin{align*}
F(x)=\int_{\mathbf R^n}E(x-\uly)F(\uly) d\uly,
\end{align*}
where $E(x)=\frac{1}{2\sigma_n}\frac{\overline x}{|x|^{n+1}}$ is the Cauchy kernel, and $F(\uly)$ is the NTBL function of $F$.
\end{prop}
When $1<p<\infty,$ one has
\begin{prop}[see e.g. \cite{Mitrea}]\label{boundedness-Cauchy}
For $f\in L^p(\mathbf R^n),1<p<\infty,$ there exists a constant $L>0$ such that
\begin{align*}
||C_{x_0}(f)||_{L^p}\leq L||f||_{L^p},
\end{align*}
where $$C_{x_0}(f)(\underline{x})=\int_{\mathbf R^n}E(x-\uly)f(\uly) d\uly,\quad x=x_0+\underline{x},x_0>0,$$ and $L$ is independent of $x_0$ and $f.$ As a consequence, as a function of $x,$ the function $C_{x_0}(f)(\ulx)$ belongs to $H^p(\mathbf R^{n+1}_+).$

Moreover,
\begin{align*}
\lim_{x_0\to 0+} C_{x_0}(f)(\ulx)=\frac{1}{2}(f(\ulx)+Hf(\ulx)), \quad a.e.\ \ulx\in \mathbf R^n.
\end{align*}
\end{prop}
There are parallel results for the lower-half of the space $\mathbf R^{n+1}.$ If we apply Proposition \ref{boundedness-Cauchy} to the NTBL of a Hardy space function $F,$ then we conclude, from the last relation, $H^2=I.$

Let $\Psi(\mathbf R^n)$ be the Clifford algebra-valued Schwartz space, whose elements are given by
\begin{align*}
\psi(\ulxi)=\sum_{T}\psi_T(\ulxi) \bfe_T,
\end{align*}
where $\psi_T$ are in the Schwartz space $S(\mathbf R^n).$ Denote by $\Psi^\pm(\mathbf R^n)$ the subclasses of $\Psi(\mathbf R^n)$ consisting of the Clifford algebra-valued Schwartz functions of, respectively, the forms
\begin{align*}
\psi(\ulxi)=\psi(\ulxi)\chi_\pm(\ulxi),
\end{align*}
where $\psi(\ulxi)$ takes the zero value in some neighborhood of the origin.
It is easy to show that the direct sum $\Psi^+(\mathbf R^n)\oplus\Psi^-(\mathbf R^n)$ is dense in $\Psi(\mathbf R^n)$.
  \def\ux{\underline{x}}
  \def\uxi{\underline{\xi}}

Denote by $\Psi^\pm_{scalar}(\mathbf R^n)$ the subclasses of $\Psi^\pm(\mathbf R^n)$, whose elements are of the form
\begin{align*}
\psi(\ulxi)=\widetilde \psi_0(\ulxi) \chi_\pm(\ulxi) ,
\end{align*}
where $\widetilde \psi_0$ is scalar-valued.

\noindent For more information on the Clifford algebra-valued distribution theory, we refer to, e.g., \cite{Brackx-Delanghe-Sommen} and \cite{DMitrea}.

\section{Main Results}
Our main results are as follows.
\begin{theorem}\label{Necessary-condition}
For $F\in H^p(\mathbf R^{n+1}_+),1\leq p\leq \infty,$ we have
\begin{align*}
(\hat F,\psi)=(F,\hat\psi)=\int_{\mathbf R^n}\hat \psi(\ulx)F(\ulx)d\ulx=0,
\end{align*}
where $\psi\in \Psi^-(\mathbf R^n).$
\end{theorem}
Conversely, we have
\begin{theorem}\label{Sufficient-condition}
Let $F\in L^p(\mathbf R^n),1\leq p\leq\infty.$ If $(\hat F,\psi)=0$ for $\psi\in \Psi^-(\mathbf R^n)$ ($\hat F=\chi_+\hat F$ if $1\leq p\leq 2$), then $F(\ulx)$ is the NTBL function of some $F(x)\in H^p(\mathbf R^{n+1}_+).$
\end{theorem}

In particular, we have
\begin{coro}\label{representation-formula}
$F\in H^p(\mathbf R^{n+1}_+),1\leq p\leq 2$ if and only if $F\in L^p(\mathbf R^n)$ and $\hat F(\ulxi)=\chi_+(\ulxi)\hat F(\ulxi),$ and
\begin{align*}
F(x)=\int_{\mathbf R^n}e^+(x,\ulxi)\hat F(\ulxi)d\ulxi=\int_{\mathbf R^n} E(x-\uly)F(\uly)d\uly.
\end{align*}
\end{coro}
\begin{remark}
Theorem \ref{Necessary-condition} and Theorem \ref{Sufficient-condition} extend Proposition \ref{stein p65} and Proposition \ref{stein p221 th6}, respectively, to the non-vector-valued Hardy spaces cases. Moreover, the sufficiency and the necessity of the result for the case $p=\infty$ are not proved in either the vector-valued or the Clifford algebra-valued Hardy spaces in literatures.
\end{remark}

\medskip

\noindent{\textbf{Proof of Theorem \ref{Necessary-condition}:}}
We first prove the result for $1<p<\infty.$ Mainly using Proposition \ref{Cauchy-formula} and Proposition \ref{boundedness-Cauchy}, we have
\begin{align}\label{ineq1}
\begin{split}
|(F,\hat \psi )| &=\lim_{x_0\to 0}|(F(\cdot+x_0),\hat \psi)|\\
&=\lim_{x_0\to 0} |(C_{x_0}(F),\hat\psi)| \\
&\leq\overline{\lim_{x_0\to 0}} |(C_{x_0}(F-G),\hat\psi)| + \overline{\lim_{x_0\to 0}}|(C_{x_0}(G),\hat\psi )|\\
&\leq C_n \overline{\lim_{x_0\to 0}} ||C_{x_0}(F-G)||_{L^p} ||\hat\psi||_{L^q} + \overline{\lim_{x_0\to 0}}|( C_{x_0}(G),\hat\psi )|\\
&\leq C_n\overline{\lim_{x_0\to 0}} L||F-G||_{L^p}||\hat\psi||_{L^q} + \overline{\lim_{x_0\to 0}}|( C_{x_0}(G),\hat\psi )|,
\end{split}
\end{align}
where $q=\frac{p}{p-1},$ $L$ is given in Proposition \ref{boundedness-Cauchy}, and $C_n$ is a constant depending on $n$.
\def\uy{\underline{y}}\def\uz{\underline{z}}

For any $\epsilon>0,$ $G$ is chosen as a smooth function with compact support such that
\begin{align*}
||F-G||_{L^p}<\epsilon.
\end{align*}
Since $C_{x_0}(G)\in H^2(\mathbf R^{n+1}_+),$ there follows
\begin{eqnarray*}
C_{x_0}(G)(\ulx)&=&\int_{\bRn}E(x_0+\ulx-\uly)G(\uly)d\uly\\
&=& \int_{\bRn}E(\frac{x_0}{2}+\ulx-\uly)C_{\frac{x_0}{2}}(G)(\uly)d\uly\\
&=& \int_{\bRn}e^{2\pi i\langle \ulx,\ulxi\rangle}e^{-2\pi\frac{x_0}{2}|\uxi|}\chi_+(\uxi)(C_{\frac{x_0}{2}}(G))^\wedge(\uxi)d\uxi.
\end{eqnarray*}
This shows that
\[ (C_{x_0}(G))^\wedge(\ulxi)=e^{-2\pi\frac{x_0}{2}|\ulxi|}\chi_+(\ulxi)(C_\frac{x_0}{2}(G))^\wedge(\ulxi).\]
Therefore,
\[ (C_{x_0}(G)(\ulx),\hat\psi(\ulx) )=((C_{x_0}(G))^\wedge(\ulxi),\psi(\ulxi) )=(e^{-2\pi \frac{x_0}{2}|\ulxi|}\chi_+(\ulxi)(C_\frac{x_0}{2}(G))^\wedge(\ulxi),\psi(\ulxi))=0.\]

Hence, through (\ref{ineq1}), for any $\epsilon>0,$ we have
\begin{align*}
|(F,\hat \psi )|\leq C_nL\epsilon||\hat\psi||_{L^q}+0,
\end{align*}
which shows $(F,\hat\psi)=0.$

Now we deal with the case $p=1.$ We use the following Lemma (for the complex analysis setting, see also e.g. \cite{Deng-Qian}).

\begin{lemma}\label{lifting up}
If $F\in H^1(\mathbf R^{n+1}_+),$ then $F_{x_0}(\underline{\cdot})=F(x_0+\underline{\cdot})\in L^2(\bRn).$
\end{lemma}
\noindent{\bf Proof} Since $F$ is left-monogenic, by the mean-value theorem, we have
\begin{align*}
F(x_0+\ulx) = \frac{1}{V_\frac{x_0}{2}}\int_{B_{x}(\frac{x_0}{2})}F(y_0+\uly)dy,
\end{align*}
where $B_{x}(\frac{x_0}{2})$ is the ball centered at $x$ with radius $\frac{x_0}{2}$, $V_\frac{x_0}{2}$ is the volume of $B_x(\frac{x_0}{2})$. Then we have
\begin{align*}
|F(x_0+\ulx)| &\leq \frac{1}{V_\frac{x_0}{2}}\int_{B_{x}(\frac{x_0}{2})}|F(y_0+\uly)| dy\\
&\leq \frac{C}{x_0^{n+1}} \int_{\frac{x_0}{2}}^\frac{3x_0}{2}\int_{\mathbf R^n}|F(y_0+\uly)| d\uly dy_0\\
&\leq \frac{C}{x_0^n}\sup_{y_0>0}\int_{\mathbf R^n}|F(y_0+\uly)| d\uly,
\end{align*}
where $C$ is a constant.
This implies that $F_{x_0}(\ulx)\in L^\infty(\mathbf R^n).$ Then, we have
\begin{align*}
\int_{\mathbf R^n}|F(x_0+\ulx)|^2d\ulx &=\int_{\mathbf R^n} |F(x_0+\ulx)| |F(x_0+\ulx)|d\ulx\\
&\leq \sup_{\ulx\in \mathbf R^n}|F(x_0+\ulx)| \int_{\mathbf R^n}|F(x_0+\ulx)| d\ulx<\infty.
\end{align*}
The proof of the lemma is complete.
\medskip

Now we continue to prove Theorem \ref{Necessary-condition}. Let $F\in H^1(\mathbf R^{n+1}_+), x_0>0.$ Then we have
$$(F, \hat{\psi})=(F-F_{x_0},\hat{\psi})+(F_{x_0},\hat{\psi}).$$
For any $\epsilon>0,$ for a suitably chosen $x_0>0,$ we have
\[|(F-F_{x_0},\hat{\psi})|\leq C_n\|F-F_{x_0}\|_1 \|\hat{\psi}\|_\infty \leq C_n\epsilon \|\hat{\psi}\|_\infty.\]
On the other hand,
as in the proof for the case $1<p<\infty,$
\[ (F_{x_0}, \hat{\psi})=0.\]
Hence
\[ |(F, \hat{\psi})|\leq C_n\epsilon \|\hat{\psi}\|_\infty\]
for arbitrary $\epsilon >0.$ That shows $(F, \hat{\psi})=0.$

For $p=\infty,$ we will mainly adopt the technique used in \cite{QXYYY}.
In Lemma \ref{formula-infty}, we will show that for $F\in H^\infty(\mathbf R^{n+1}_+)$,
\begin{align*}
F(\ulx)=HF(\ulx) +c,
\end{align*}
where $c$ is a constant. It is easily shown that
\begin{align*}
(c,\hat \psi)=0,\quad \psi\in \Psi^-(\mathbf R^n).
\end{align*}
Hence, in the following we only need to consider
\begin{align*}
F(\ulx)=\frac{1}{2}(I+H)F(\ulx)=\frac{1}{2}\sum_{T}(I+H)f_T(\ulx)\bfe_T.
\end{align*}
We first consider that $\psi\in \Psi^-_{scalar}(\mathbf R^n).$ We have $\psi(\ulxi)=\widetilde\psi_0(\ulxi)\chi_-(\ulxi)=\psi_0(\ulxi)+\sum_{j=1}^n\psi_j(\ulxi)\bfe_j,$ where $\widetilde\psi_0$ is a scalar-valued Schwartz function. In particular, we note that
\begin{align*}
\hat \psi (\ulx)&=\int_{\mathbf R^n}\widetilde\psi_0(\ulxi)\frac{1}{2}(1-i\frac{\ulxi}{|\ulxi|})e^{-2\pi i \langle\ulx,\ulxi\rangle}d\ulxi\\
&=\int_{\mathbf R^n}\widetilde\psi_0(-\ulxi)\frac{1}{2}(1+i\frac{\ulxi}{|\ulxi|})e^{2\pi i \langle\ulx,\ulxi\rangle}d\ulxi,
\end{align*}
which means that $\hat\psi(\ulx)$ is the NTBL function of functions in $\mathcal H^1(\mathbf R^{n+1}_+).$
In the following we first accept, and use this property to prove that $(F,\hat \psi)=0$ for $\psi\in \Psi^-(\mathbf R^n).$
 Since $\hat\psi\in \mathcal H^1(\mathbf R^{n+1}_+)$ and $(I+H)f_T(\ulx)\in BMO(\mathbf R^n),$
$
((I+H)f_T(\ulx),\hat\psi)
$
is well-defined. As shown in \cite[page 146]{Fefferman-Stein}, we have
\begin{align*}
\int_{\mathbf R^n}\hat\psi(\ulx)R_j(f_T)(\ulx) d\ulx &=\sum_{k=0}^n\int_{\mathbf R^n}\bfe_k\hat \psi_k(\ulx)R_j(f_T)(\ulx)d\ulx\\
&=-\sum_{k=0}^n \int_{\mathbf R^n}\bfe_k R_j(\hat\psi_k)(\ulx)f_T(\ulx)d\ulx\\
&=-\int_{\mathbf R^n}R_j(\hat\psi) f_T(\ulx)d\ulx.
\end{align*}
Hence we have
\begin{align*}
\int_{\mathbf R^n}\hat\psi(\ulx)(f_T(\ulx)-\sum_{j=1}^n R_j(f_T)(\ulx)\bfe_j)d\ulx
&=\int_{\mathbf R^n}(\hat\psi(I+\sum_{j=1}^n R_j\bfe_j))(\ulx)f_T(\ulx) d\ulx. \\
\end{align*}
One can easily show that $\hat\psi(I+\sum_{j=1}^n R_j \bfe_j)=0.$ In fact, since the Fourier multiplier of $R_j$ is $-i\frac{\xi_j}{|\ulxi|}$, we have
\begin{align*}
(\hat\psi(I+\sum_{j=1}^n R_j\bfe_j))^\vee(\ulxi) =\psi(\ulxi)(1+i\sum_{j=1}^n \frac{\xi_j \bfe_j}{|\ulxi|})=\widetilde\psi_0(\ulxi)\frac{1}{2}(1-i\frac{\ulxi}{|\ulxi|})(1+i\frac{\ulxi}{|\ulxi|})=0.
\end{align*}
Thus $\hat\psi(I+\sum_{j=1}^n R_j\bfe_j)=0,$ and hence $((I+H)f_T(\ulx),\hat \psi)=0$ for $\psi\in \Psi^-_{scalar}(\mathbf R^n).$ Consequently, $(F,\hat\psi)=0$ for $\psi\in \Psi^-_{scalar}(\mathbf R^n).$ For $\psi\in \Psi^-(\mathbf R^n),$ we have
\begin{align*}
(F,\hat\psi)=\sum_{S}\bfe_S(F,(\psi_S\chi_-)^\wedge).
\end{align*}
Then, for each $S,$ we have
$$
(F,(\psi_S\chi_-)^\wedge)=0
$$
by applying the above argument. We have $(F,\hat\psi)=0$ for $\psi\in \Psi^-(\mathbf R^n).$

\quad \hfill$\Box$\vspace{2ex}

\noindent{\textbf{Proof of Theorem \ref{Sufficient-condition}:}}
We first consider the case $p=1.$ Define
\begin{align*}
\Phi(x_0+\ulx)=\int_{\mathbf R^n}e^+(x,\ulxi)\hat F(\ulxi)d\ulxi
\end{align*}
and
\begin{align*}
G_{x_0}(\ulx)=G(x_0+\ulx)=\int_{\mathbf R^n}P_{x_0}(\ulx-\uly)F(\uly)d\uly.
\end{align*}

We then have
\begin{align*}
\Phi(x_0+\ulx) = \int_{\mathbf R^n}e^+(x,\ulxi)\hat F(\ulxi)d\ulxi= \int_{\mathbf R^n}P_{x_0}(\ulx-\uly)F(\uly)d\uly= G(x_0+\ulx).
\end{align*}
 We also note that $G_{x_0}\in L^1(\mathbf R^n)$ since
\begin{align*}
||G_{x_0}||_{L^1}\leq C_n||F||_{L^1}||P_{x_0}||_{L^1}<\infty.
\end{align*}
By interchanging the derivatives with the integral, we can show that $\Phi$ is left-monogenic since $e^+(x,\ulxi)$ is monogenic.
Thus $\Phi(x)\in H^1(\mathbf R^{n+1}_+).$

Next we consider the case for $L^p(\mathbf R^{n}), 1<p\leq\infty.$ By the assumption, we have
\begin{align}\label{eq1}
0=(\hat F,\psi)=( (F^+)^\wedge+(F^-)^\wedge,\psi)=((F^-)^\wedge,\psi),\quad \text{for all }\psi\in \Psi^-(\mathbf R^n),
\end{align}
where $F^+=\frac{1}{2}(I+H)F$ and $F^-=\frac{1}{2}(I-H)F,$ and $((F^+)^\wedge,\psi)=0$ follows from the proof of Theorem \ref{Necessary-condition} (Note that this argument only works for $1<p\leq \infty,$ not including $p=1$). Consequently, we can have, for $\varphi$ being a scalar-valued Schwartz function taking zero in some neighborhood of the origin,
\begin{align*}
((F^-)^\wedge,\varphi)=((F^-)^\wedge,\varphi\chi_+)+((F^-)^\wedge,\varphi\chi_-)=0,
\end{align*}
where $((F^-)^\wedge,\varphi\chi_+)=0$ follows from the proof of Theorem \ref{Necessary-condition}, and $((F^-)^\wedge,\varphi\chi_-)=0$ is given by (\ref{eq1}).
 This implies that $(F^-)^\wedge$ is either zero or a distribution with support at the origin. For the latter case, $(F^-)^\wedge$ has to be a finite linear combination of the partial derivatives of the Dirac delta function (see e.g. \cite{Hor,DMitrea}), which contradicts to $F\in L^p(\mathbf R^n),1<p<\infty.$ Thus, for $1<p<\infty,$ $(F^-)^\wedge=0$, and then $F^-=0.$ For $p=\infty,$ $F^-$ is a constant $c.$ Then, for $1<p<\infty,$ we have $F=F^+,$ and for $p=\infty,$ $F=F^++c$ being the NTBL function of some functions in $H^p(\mathbf R^{n+1}_+).$
\quad \hfill$\Box$\vspace{2ex}

\medskip

\begin{lemma}\label{formula-infty}
For $F\in H^\infty(\mathbf R^{n+1}_+),$ we have
\begin{align}\label{Cauchy-infty}
F(\ulx)=HF(\ulx)+c,
\end{align}
where $c$ is a constant.
\end{lemma}
\begin{proof}
Denote by $S_0(\mathbf R^n),$ where $\phi\in S_0(\mathbf R^n)$ is a scalar-valued Schwartz function such that $\hat\phi$ takes the zero value in some neighborhood of the origin.
To show (\ref{Cauchy-infty}), it suffices to show that
\begin{align}\label{tem-ineq1}
((I-H)F(\ulx), \phi(\ulx))=0
\end{align}
holds for all $\phi\in S_0(\mathbf R^n).$ In fact, we have
$$((I-H)F(\ulx),\phi(\ulx))=(((I-H)F)^\vee(\ulxi),\hat\phi(\ulxi)),$$ which means that $((I-H)F)^\vee$ is either zero or a distribution with support at the origin. Thus $(I-H)F(\ulx)=c,$ and hence $F(\ulx)=HF(\ulx)+c,$ where $c$ is a constant.

Let $F=\sum_{T}f_T\bfe_T,$ and $\psi(\ulx)=\left(\frac{\hat\phi(\ulxi)}{2\pi|\ulxi|}\right)^\vee(\ulx).$ Moreover, we have
\begin{align*}
-2\pi |\ulxi|e^{-2\pi x_0|\ulxi|}\hat\psi(\ulxi)=-e^{-2\pi x_0|\ulxi|}\hat\phi(\ulxi),
\end{align*}
 and hence,
 \begin{align*}
 (\partial_0P_{x_0}*\psi)(\ulx)=-P_{x_0}*\phi(\ulx).
 \end{align*}

Note that for $\phi\in S_0(\mathbf R^n),$ $\psi$ is still a Schwartz function. For each $T,$ we have
\begin{align}\label{ineq2}
\begin{split}
&((I+\sum_{j=1}^n \bfe_jR_j)f_T(x_0+\ulx),-\phi(\ulx))\\
&=((I+\sum_{j=1}^n\bfe_jR_j)f_T(x_0+\ulx), (-2\pi|\ulxi|\hat\psi(\ulxi))^\vee(\ulx))\\
&=(f_T(x_0+\ulx), (-2\pi|\ulxi|\hat\psi(\ulxi))^\vee(\ulx)(I-\sum_{j=1}^n R_j\bfe_j))\\
&=(f_T(x_0+\ulx), [(-2\pi|\ulxi|\hat\psi(\ulxi))(1+\sum_{j=1}^n i\frac{\xi_j\bfe_j}{|\ulxi|})]^\vee(\ulx))\\
&=(f_T(x_0+\ulx), -\phi(\ulx))+(f_T(x_0+\ulx), [\sum_{j=1}^n -2\pi i\xi_j\hat\psi(\ulxi)\bfe_j]^\vee(\ulx))\\
&=(f_T(\ulx),-P_{x_0}*\phi(\ulx) )+(f_T(x_0+\ulx), [\sum_{j=1}^n -2\pi i\xi_j\hat\psi(\ulxi)\bfe_j]^\vee(\ulx))\\
&=(f_T(\ulx),(\partial_0P_{x_0}*\psi)(\ulx) )+(f_T(x_0+\ulx), -\sum_{j=1}^n \partial_j\psi(\ulx)\bfe_j).
\end{split}
\end{align}
In the second equality we have used the result proved in \cite[page 146]{Fefferman-Stein}, while the other equalities follow from the properties of the Fourier transform and the Fourier multiplier of the Riesz transformations.
Since $\psi$ is a Schwartz function, for any $\epsilon>0,$ we can find $\widetilde\psi$, a $C^\infty$- function with compact support, such that $||\psi-\widetilde \psi||_{L^1}<\epsilon$ and $||\partial_j\psi-\partial_j\widetilde \psi||_{L^1}<\epsilon,1\leq j\leq n.$ We also note that
\begin{align*}
\partial_0P_{x_0}(\ulx)=\frac{1}{\sigma_n}\frac{\partial}{\partial x_0}\frac{x_0}{(x_0^2+|\ulx|^2)^\frac{n+1}{2}}=\frac{1}{\sigma_n}\frac{-nx_0^2+|\ulx|^2}{(x_0^2+|\ulx|^2)^{\frac{n+1}{2}+1}}.
\end{align*}
Thus,
\begin{align*}
&|(f_T(\ulx),(\partial_0P_{x_0}*(\psi-\widetilde\psi))(\ulx) )|\\
&\leq \frac{C^\prime}{\sigma_n}\int_{\mathbf R^n}\int_{\mathbf R^n}\frac{x_0^2}{(x_0^2+|\ulx-\uly|^2)^{\frac{n+1}{2}+1}}|\psi(\uly)-\widetilde\psi(\uly)|d\uly d\ulx\\
&+\frac{C^{\prime\prime}}{\sigma_n}\int_{\mathbf R^n}\int_{\mathbf R^n}\frac{|\ulx-\uly|^2}{(x_0^2+|\ulx-\uly|^2)^{\frac{n+1}{2}+1}}|\psi(\uly)-\widetilde\psi(\uly)|d\uly d\ulx\\
&\leq C^\prime||P_{x_0}*(|\psi-\widetilde\psi|)||_{L^1}+\frac{C^{\prime\prime}}{x_0}||P_{x_0}*(|\psi-\widetilde\psi|)||_{L^1}\\
&\leq (C^\prime+\frac{C^{\prime\prime}}{x_0})||\psi-\widetilde\psi||_{L^1},\\
&\leq C\epsilon.
\end{align*}
where $C^\prime,C^{\prime\prime}$ and $C$ are constants.
Hence,
\begin{align*}
(\ref{ineq2})&=((\partial_0P_{x_0}*f_T)(\ulx),\widetilde\psi(\ulx) )+(\sum_{j=1}^n\partial_jf_T(x_0+\ulx)\bfe_j,\widetilde\psi(\ulx))\\
&+(f_T(\ulx),\partial_0P_{x_0}*(\psi-\widetilde\psi)(\ulx))+(f_T(x_0+\ulx), \sum_{j=1}^n-(\partial_j\psi(\ulx)-\partial_j\widetilde\psi(\ulx))\bfe_j)\\
&=((\partial_0+\sum_{j=1}^n\partial_j\bfe_j)f_{T}(x_0+\ulx),\widetilde\psi(\ulx))\\
&+(f_T(\ulx),\partial_0P_{x_0}*(\psi-\widetilde\psi)(\ulx))+(f_T(x_0+\ulx), \sum_{j=1}^n-(\partial_j\psi(\ulx)-\partial_j\widetilde\psi(\ulx))\bfe_j)\\
&=(Df_T(x_0+\ulx),\widetilde\psi(\ulx))\\
&+(f_T(\ulx),\partial_0P_{x_0}*(\psi-\widetilde\psi)(\ulx))+(f_T(x_0+\ulx), \sum_{j=1}^n-(\partial_j\psi(\ulx)-\partial_j\widetilde\psi(\ulx))\bfe_j).\\
\end{align*}
Then, we have
\begin{align*}
&((I+\sum_{j=1}^n \bfe_jR_j)F(x_0+\ulx),-\phi(\ulx))\\
&=(DF(x_0+\ulx),\widetilde\psi(\ulx))\\
&+(F(\ulx),\partial_0P_{x_0}*(\psi-\widetilde\psi)(\ulx))+(F(x_0+\ulx), \sum_{j=1}^n-(\partial_j\psi(\ulx)-\partial_j\widetilde\psi(\ulx))\bfe_j).
\end{align*}
Since $F$ is left-monogenic, we have
\begin{align*}
&((I+\sum_{j=1}^n \bfe_jR_j)F(x_0+\ulx),-\phi(\ulx))\\
&=(F(\ulx),\partial_0P_{x_0}*(\psi-\widetilde\psi)(\ulx))+(F(x_0+\ulx), \sum_{j=1}^n-(\partial_j\psi(\ulx)-\partial_j\widetilde\psi(\ulx))\bfe_j).
\end{align*}
Thus, for any $\epsilon>0,$ we have
\begin{align*}
&|((I+\sum_{j=1}^n \bfe_jR_j)F(x_0+\ulx),-\phi(\ulx))|\\
&\leq|(F(\ulx),\partial_0P_{x_0}*(\psi-\widetilde\psi)(\ulx))|+|(F(x_0+\ulx), \sum_{j=1}^n-(\partial_j\psi(\ulx)-\partial_j\widetilde\psi(\ulx))\bfe_j)|\\
&\leq C^{\prime\prime\prime} (||\psi-\widetilde\psi||_{L^1}+\sum_{j=1}^n||\partial_{j}\psi-\partial_j\widetilde\psi)||_{L^1})\\
&\leq (n+1)C^{\prime\prime\prime}\epsilon,
\end{align*}
where $C^{\prime\prime\prime}$ is a constant.
Hence,
\begin{align*}
((I+\sum_{j=1}^n \bfe_jR_j)F(x_0+\ulx),-\phi(\ulx))=0.
\end{align*}
 Finally, we have
$$
((I+\sum_{j=1}^n\bfe_j R_j)F(\ulx),\phi(\ulx))=\lim_{x_0\to 0}((I+\sum_{j=1}^n \bfe_jR_j)F(x_0+\ulx),\phi(\ulx))=0
$$
for all $\phi\in S_0(\mathbf R^n).$
The proof is completed.
\end{proof}

\begin{remark}
If $F\in H^\infty(\mathbf R^{n+1}_+)$ is vector-valued, the above result becomes
\begin{align*}
F(\ulx)=f_0(\ulx)-\sum_{j=1}^n R_j(f_0)(\ulx) \bfe_j+c,
\end{align*}
which is a special case of the celebrated characterization of $BMO$ proved by Fefferman and Stein in \cite{Fefferman-Stein}.

We note that the proof of the above lemma holds for $H^p(\mathbf R^{n+1}_+),1\leq p\leq\infty.$ This means that we have given an alternative proof of the Cauchy integral formula for $F\in H^p(\mathbf R^{n+1}_+),1\leq p< \infty,$ i.e.,
\begin{align*}
F(x_0+\ulx)=\int_{\mathbf R^n}E(x-\uly)F(\uly)d\uly,
\end{align*}
and for $F\in H^\infty(\mathbf R^{n+1}_+),$
\begin{align*}
F(x_0+\ulx)=\int_{\mathbf R^n}E(x-\uly)F(\uly)d\uly+c,
\end{align*}
where $E(x-\uly)$ is revised corresponding to the revision of $R_j,1\leq j\leq n.$

When $p=2,$ we can prove $F(\ulx)=HF(\ulx)$ in the normal sense. The proof of the above lemma is motivated by the case $p=2.$ In fact, a direct computation of $DF=0$ yields
\begin{align}\label{DF}
(\partial_0 f_T(x_0+\ulx) +\sum_{j=1}^n (-1)^{l_j} \partial_jf_{T_j}(x_0+\ulx))\bfe_T=0,\quad \text{ for all }T,
\end{align}
where $T_j$ satisfies $(-1)^{l_j}\bfe_j\bfe_{T_j}=\bfe_T.$ We note that
$$l_j=N(j\cap T_j)+P(j,T_j),$$
where $N(A)=\#A$ denotes the number of elements in some set $A,$ and
$$
P(j,T_j)=\#\{k;j>k,k\in T_j\}.
$$
Taking the Fourier transform on $\ulx,$ we have
\begin{align*}
(\partial_0 \hat f_T(x_0+\ulxi) +\sum_{j=1}^n (-1)^{l_j} (2\pi i\xi_j)\hat f_{T_j}(x_0+\ulxi))\bfe_T=0.
\end{align*}
Using the fact that
\begin{align*}
\lim_{x_0\to 0}\partial_0 \hat f_T(x_0+\ulxi)=\lim_{x_0\to 0}\partial_0(P_{x_0}*f_T)^\wedge(\ulxi)=\lim_{x_0\to 0}\partial_0(e^{-2\pi x_0|\ulxi|}\hat f_T(\ulxi))=-2\pi|\ulxi|\hat f_T(\ulxi),
\end{align*}
we have
\begin{align*}
(\hat f_T(\ulxi) +\sum_{j=1}^n (-1)^{l_j} (-i\frac{\xi_j}{|\ulxi|})\hat f_{T_j}(\ulxi))\bfe_T=0,
\end{align*}
and consequently,
\begin{align*}
( f_T(\ulx) +\sum_{j=1}^n (-1)^{l_j} R_j(f_{T_j})(\ulx))\bfe_T=0.
\end{align*}
This means that $F(\ulx)=HF(\ulx).$

We also note that the system (\ref{DF}) is indeed a generalization of the conjugate harmonic system (\ref{GCR}). When $F$ is vector-valued, the system (\ref{DF}) is reduced to $(\ref{GCR})$ (see the simple argument given in \S 2). Consequently, for $H^p(\mathbf R^{n+1}_+)\ni F=f_0+\sum_{j=1}^n f_j\bfe_j,$ $F=HF$ is then reduced to $f_j=-R_j(f_0),1\leq j\leq n$, which is the classical result for the conjugate harmonic system (see Propositions \ref{stein p65}$-$\ref{stein p221 th6}).
\end{remark}

\section{Analogue in Bergman Space}
In this section we prove a representation formula for the functions in the Clifford algebra-valued Bergman space in the upper half-space as an application of our results in the previous sections. For the analogous result in the complex analysis setting, we refer to \cite{Genchev} and \cite{BG}.
\begin{defn}
Let $F(x)=\sum_{T}f_T(x)\bfe_T,$ where $x\in \mathbf R^{n+1}_+.$ If $F$ is left-monogenic on $\mathbf R^{n+1}_+,$ and satisfies
\begin{align*}
||F||_{A^p}^p=\int_{0}^\infty \int_{\mathbf R^n}|F(x_0+\ulx)|^pd\ulx dx_0<\infty,\quad 1\leq p<\infty,
\end{align*}
then we say that $F$ belongs to the Bergman space $A^p(\mathbf R^{n+1}_+).$
\end{defn}

As an application of the Fourier spectrum characterization of $H^p(\mathbf R^{n+1}_+),1\leq p<\infty,$ we have
\begin{theorem}\label{bergman-2}
For $F\in A^p(\mathbf R^{n+1}_+),1\leq p\leq 2,$ there exists a function $G\in L^q (\mathbf R^n), q=\frac{p}{p-1},$ such that
\begin{align*}
F(x)=\int_{\mathbf R^n}e^+(x,\ulxi)G(\ulxi)d\ulxi.
\end{align*}
Moreover, for $1<p\leq 2,$
\begin{align*}
\left(\int_{\mathbf R^n}\frac{|\chi_+(\ulxi)G(\ulxi)|^q}{(2\pi p|\ulxi|)^\frac{q}{p}} d\ulxi\right)^\frac{1}{q} \leq ||F||_{A^p}<\infty,
\end{align*}
and for $p=1,$
\begin{align*}
\sup_{\ulxi\in \mathbf R^n}\frac{|\chi_+(\ulxi)G(\ulxi)|}{2\pi|\ulxi|}\leq ||F||_{A^1}<\infty.
\end{align*}
\end{theorem}
\begin{theorem}\label{bergman-p}
For $F\in A^p(\mathbf R^{n+1}_+),2<p<\infty,$ we have that, for $x_0>0,$ there holds
\begin{align*}
(F(x_0+\cdot),\hat\psi)=0
\end{align*}
for $\psi\in \Psi^-(\mathbf R^n).$
\end{theorem}
\noindent{\textbf{Proof of Theorem \ref{bergman-2}}}\\
Since $F$ is left-monogenic, $|F|^p$ is subharmonic for $1\leq p\leq 2$. We thus have
\begin{align*}
|F(x)|^p\leq \frac{1}{V_{\delta}}\int_{B_{x}(\delta)}|F(y_0+\uly)|^p dy,
\end{align*}
where $B_{x}(\delta)$ is the ball centered at $x$ with radius $0<\delta<x_0$ (for instance, let $\delta=\frac{x_0}{2}$), $V_\delta$ is the volume of $B_{x}(\delta).$ Then we have
\begin{align*}
|F(x)|^p &\leq \frac{C}{\delta^{n+1}}\int_{x_0-\delta}^{x_0+\delta}\int_{|\ulx-\uly|<\delta}|F(y_0+\uly)|^p d\uly dy_0\\
&\leq \frac{C}{\delta^{n+1}}\int_{0}^{\infty}\int_{|\ulx-\uly|<\delta}|F(y_0+\uly)|^p d\uly dy_0,
\end{align*}
and then,
\begin{align*}
\int_{\mathbf R^n}|F(x_0+\ulx)|^p d\ulx
&\leq \frac{C}{\delta^{n+1}}\int_{\mathbf R^n}\int_{0}^{\infty}\int_{|\ulx-\uly|<\delta}|F(y_0+\uly)|^p d\uly dy_0 d\ulx\\
&=\frac{C}{\delta^{n+1}}\int_{0}^{\infty}\int_{\mathbf R^n}\int_{\mathbf R^n}\chi_{B_{\uly}(\delta)}(\ulx)d\ulx|F(y_0+\uly)|^p d\uly dy_0\\
&=\frac{C^\prime\delta^n}{\delta^{n+1}}\int_0^\infty \int_{\mathbf R^n} |F(y_0+\uly)|^p d\uly dy_0\\
&=\frac{C^\prime}{\delta}\int_0^\infty \int_{\mathbf R^n} |F(y_0+\uly)|^p d\uly dy_0
\end{align*}
where the second equality used Fubini's theorem, and $C$ and $C^\prime$ are constants. The above inequality implies that $F_{y_0}(x)=F(y_0+x)\in H^p(\mathbf R^{n+1}_+)$ for $F\in A^p(\mathbf R^{n+1}_+)$ and $y_0>0.$ Consequently, we have
\begin{align*}
F_{y_0}(x)&=\int_{\mathbf R^n} e^+(x,\ulxi)\hat F_{y_0}(\ulxi)d\ulxi
\end{align*}
and
\begin{align*}
F_{y_0}(x)=\int_{\mathbf R^n}P_{x_0}(\ulx-\uly)F_{y_0}(\uly)d\uly.
\end{align*}
Thus
\begin{align*}
\hat F_{y_0+x_0}(\ulxi)=e^{-2\pi x_0|\ulxi|}\hat F_{y_0}(\ulxi)
\end{align*}
and hence
\begin{align*}
e^{2\pi(y_0+x_0)|\ulxi|}\hat F_{y_0+x_0}(\ulxi)=e^{2\pi y_0|\ulxi|}\hat F_{y_0}(\ulxi).
\end{align*}
Therefore, if we let $G(\ulxi)=e^{2\pi y_0|\ulxi|}\hat F_{y_0}(\ulxi),$ which is independent of $y_0,$ then we have
\begin{align*}
F_{y_0}(x)&=\int_{\mathbf R^n}e^+(x,\ulxi)e^{-2\pi y_0|\ulxi|}G(\ulxi)d\ulxi\\
&=\int_{\mathbf R^n}e^+(y_0+x,\ulxi)G(\ulxi)d\ulxi.
\end{align*}
Consequently,
\begin{align*}
F(x)&=\int_{\mathbf R^n}e^+(x,\ulxi)G(\ulxi)d\ulxi.
\end{align*}
Moreover, by Hausdorff-Young's inequality, we have, for $1<p\leq 2,$
\begin{align*}
\left(\int_{\mathbf R^n}|\chi_+(\ulxi)G(\ulxi)|^q e^{-2\pi x_0q|\ulxi|} d\ulxi\right)^\frac{1}{q}\leq \left(\int_{\mathbf R^n}|F(x_0+\ulx)|^p d\ulx\right)^\frac{1}{p},
\end{align*}
and for $p=1,$
\begin{align*}
\sup_{\ulxi\in \mathbf R^n}|\chi_+(\ulxi)G(\ulxi)| e^{-2\pi x_0|\ulxi|}\leq\int_{\mathbf R^n}|F(x_0+\ulx)| d\ulx.
\end{align*}
For $1<p\leq 2,$ by Minkowski's inequality,
\begin{align*}
\left(\int_{\mathbf R^n} \left(\int_0^\infty |\chi_+(\ulxi)G(\ulxi)|^p e^{-2\pi x_0p|\ulxi|}dx_0\right)^\frac{q}{p}d\ulxi\right)^\frac{p}{q}&\leq \int_{0}^\infty \left(\int_{\mathbf R^n}|\chi_+(\ulxi)G(\ulxi)|^q e^{-2\pi x_0q|\ulxi|}d\ulxi\right)^\frac{p}{q} dx_0\\
&\leq\int_0^\infty\int_{\mathbf R^n}|F(x_0+\ulx)|^p d\ulx dx_0,
\end{align*}
and thus
\begin{align*}
\left(\int_{\mathbf R^n}\frac{|\chi_+(\ulxi)G(\ulxi)|^q}{(2\pi p|\ulxi|)^\frac{q}{p}}\right)^\frac{1}{q}\leq \left(\int_0^\infty\int_{\mathbf R^n}|F(x_0+\ulx)|^p d\ulx dx_0\right)^\frac{1}{p}<\infty.
\end{align*}
For $p=1,$
\begin{align*}
\sup_{\ulxi\in \mathbf R^n}\frac{|\chi_+(\ulxi)G(\ulxi)|}{2\pi|\ulxi|}\leq\int_0^\infty\int_{\mathbf R^n}|F(x_0+\ulx)| d\ulx dx_0<\infty.
\end{align*}

\quad \hfill$\Box$\vspace{2ex}

\noindent{\textbf{Proof of Theorem \ref{bergman-p}}}\\
As in the proof of Theorem \ref{bergman-2}, we can show that $F_{y_0}(x)\in H^p(\mathbf R^{n+1}_+)$ for $F\in A^p(\mathbf R^{n+1}_+),2<p<\infty$ and $y_0>0.$ Then, by Theorem \ref{Necessary-condition} we complete the proof.
\quad \hfill$\Box$

\bigskip

\vspace{2ex}

\begin{thebibliography}{10}

\bibitem{BG} D. B$\acute{e}$koll$\acute{e}$ and A. Bonami, {\it Hausdorff-Young Inequalities for Functions in Bergman Spaces on Tube Domains}, Proceedings of the Edinburgh Mathematical Society, {\bf 41} (1998), pages: 553--566.

\bibitem{Brackx-Delanghe-Sommen} F. Brackx, R. Delanghe and F. Sommen, {\it Clifford Analysis,} Research Notes in Mahtematics 76, Pitman Advanced Publishing Company, Boston London, Melbourne, 1982.
\bibitem{Deng-Han} D.-G. Deng and Y.-S. Han, {Theory of $H^p$ spaces (in Chinese),} Peking University Press, Peking, 1992.

\bibitem{Deng-Qian} G.-T. Deng and T. Qian, {\it Rational Approximation of functions in Hardy spaces,} Complex Analysis and Operator Theory, {\bf 10} (2016), pages: 903--320.

\bibitem{Duren} P. Duren, {\it Theory of $H^p$ Spaces,} Pure and Applied Mathematics, Volume 38, Academic Press, New York and Lodon, 1970.

\bibitem{Genchev} T. G. Genchev, {\it Paley-Wiener Type Theorems for Functions in Bergman Spaces over Tube Domains,} Journal of Mathematical Analysis and Applications, {\bf 118} (1986), pages: 496--501.

\bibitem{Fefferman-Stein}
C. Fefferman and E. M. Stein, {\it {$H^{p}$} spaces of several variables},
 Acta Math. {\bf 129} (1972), pages: 137--193.

\bibitem{Garrigos} G. Garrig\'os, {\it Generalized Hardy Spaces on Tube Domains over cones,} Colloq. Math. {\bf 90} (2001), pages: 213--251.
\bibitem{Garnett}  J. B. Garnett, {\it Bounded analytic functions}, Academic Press, 1987.
\bibitem{Gilbert-Murray} J. Gilbert and M. Murray, {\it Clifford Algebras and Dirac Operators in Harmonic Analysis,} Cambridge University Press, 1991.

\bibitem{Grafakos} L. Grafakos, {\it Classical Fourier Analysis (Second Edition),} Graduate Texts in Mahtematics 249, Springer, 2008.
\bibitem{Hor} L. H\"ormander, {\it The Analysis of Linear Partial Differential Operator, I. Distribution theorey and Fourier analysis}, 2nd Ed., Springer-Berlag, Berlin, 2003.

\bibitem{Kou-Qian} K.-I. Kou and T. Qian, {\it The Paley-Wiener Theorem in $\mathbf R^n$ with the Clifford Analysis Setting,} Journal of Functional Analysis, {\bf 189} (2002), pages: 227--241.
\bibitem{Li-Deng-Qian} H.-C. Li, G.-T. Deng and T. Qian, {\it Fourier Spectrum Characterizations of Hp Spaces on Tubes over Cones for $1\leq p\leq \infty$,} Complex Analysis and Operator Theory, {\bf 12} (2018), pages: 1193--1218.

\bibitem{Li-Mc-Qian} C. Li, A. McIntosh and T. Qian, {\it Clifford Algebras, Fourier Transform and Singular Convolution Operators on Lipschitz Surfaces,} Revista Mathematica Iberoamericana, {\bf 10} (1994), pages: 665--721.

\bibitem{Mitrea} M. Mitrea, {\it Clifford Wavelets, Singular Integrals and Hardy Spaces,} Lecture Note in Mathematics, Springer-Verlag, Berlin Heidelberg, 1994.

\bibitem{DMitrea} D. Mitrea, {\it Distributions, Partial Differential Equations, and Harmonic Analysis,} Universitext, Springer, New York, 2013.

\bibitem{Q1}  T. Qian, {\it Characterization of boundary values of functions in Hardy spaces with applications in signal analysis},  J. Integral Equations Appl. {\bf 17} (2005), pages: 159--198.
\bibitem{QXYYY}  T. Qian, Y. S. Xu, D. Y.  Yan, L. X. Yan,  B. Yu, {\it Fourier spectrum characterization of Hardy spaces and applications}, Proceedings of the American Mathematical Society, {\bf 137} (2009), pages: 971--980.


\bibitem{SW}  E. M. Stein,  G. Weiss, {\it Introduction to fourier analysis on Euclidean spaces}, Princeton University Press, 1971.


\bibitem{Stein2} E. M. Stein, {\it Singular Integrals and
Differentiability Properties of Functions}, Princeton University
Press, Princeton, New Jersey, 1970.

\bibitem{Uchiyama} A. Uchiyama, {\it A Constructive Proof of the Fefferman-Stein Decomposition of $BMO(\mathbf R^n)$,} Acta Math. {\bf 148} (1982), pages: 215--241.


\end{thebibliography}
\end{document}